\documentclass[11pt,a4paper]{article}
\usepackage{amsmath,amssymb,amsthm}
\usepackage{geometry}
\usepackage[hidelinks]{hyperref}
\numberwithin{equation}{section}
\allowdisplaybreaks[2]
\theoremstyle{plain}
\newtheorem{theorem}{Theorem}[section]
\newtheorem{proposition}[theorem]{Proposition}
\newtheorem{lemma}[theorem]{Lemma}

\theoremstyle{remark}
\newtheorem{remark}[theorem]{Remark}
\newcommand{\R}{\mathbb{R}}
\providecommand{\C}{}
\renewcommand{\C}{\mathbb{C}}
\newcommand{\Cp}{\mathbb{C}^{+}}
\newcommand{\Cm}{\mathbb{C}^{-}}
\newcommand{\Pp}{\mathcal{P}}
\newcommand{\bp}{\boxplus}
\newcommand{\supp}{\operatorname{supp}}
\newcommand{\im}{\operatorname{Im}}

\newcommand{\ind}{\mathbf{1}}
\newcommand{\abs}[1]{\left\lvert #1\right\rvert}

\title{The Berry--Esseen Estimate \\in the Free Central Limit Theorem}
\author{
Makoto Maejima\thanks{
Professor Emeritus, Keio University, Japan. \texttt{maejima@math.keio.ac.jp}
}
\and
Noriyoshi Sakuma\thanks{
Department of Mathematics, Graduate School of Science,
The University of Osaka, 1-1 Machikaneyama-cho, Toyonaka, Osaka 560-0043, Japan.\\
Corresponding author:\texttt{sakuma@math.sci.osaka-u.ac.jp}
}
}
\date{\today}

\begin{document}
\maketitle

\begin{abstract}
We consider sums of freely independent self-adjoint random variables that are not necessarily identically distributed. 
Let $\mu_j$ denote the distribution of the $j$th summand. 
We assume that they have mean zero and finite absolute moments of order $2+\delta$, where $0<\delta\le 1$. 
Let $\Delta$ denote the Kolmogorov distance, let $\mu^{(n)}$ be the distribution of the normalized partial sum, let $\omega$ be the standard semicircle law, and let $B_n^2$ be the variance of the partial sum.
The purpose of this paper is to prove the Berry--Esseen estimate in the free central limit theorem.
Namely, there exists an absolute constant $C>0$ such that, for every $0<\delta\le 1$,
\[
 \Delta(\mu^{(n)},\omega)
 \le \frac{C}{B_n^{2+\delta}}\sum_{j=1}^n
 \int_{\R}|x|^{2+\delta}\,\mu_j(dx),
\]
Our result not only improves several known estimates for general non-identically distributed random variables, but also establishes exactly the same Berry--Esseen estimate as in classical probability theory.
The proof combines truncation, a quantitative estimate for the
$R$-transform, a stability analysis of a perturbed semicircle equation, and a Bai-type smoothing inequality.
\end{abstract}



\section{Introduction and main result}\label{sec:intro}

Let $\Pp(\R)$ be the set of all Borel probability measures on $\R$. Let $\mu_1,\ldots,\mu_n\in\Pp(\R)$ satisfy
\[
 \int_{\R}x\,\mu_j(dx)=0,
 \qquad
 \sigma_j^2:=\int_{\R}x^2\,\mu_j(dx)\in(0,\infty).
\]
Denote
\begin{align*}
&\beta_r(\mu):=\int_{\R}|x|^r\,\mu(dx),\quad \mu\in\mathcal{P}(\R),\quad r\ge 1,\\
& B_n^2:=\sum_{j=1}^n\sigma_j^2,
 \qquad
 \mu_{n,j}:=D_{B_n^{-1}}\mu_j,
 \qquad
 \mu^{(n)}:=\mu_{n,1}\bp\cdots\bp\mu_{n,n},
\end{align*}
where $D_c\mu$ is the push-forward of $\mu$ under the map $x\mapsto cx$ and the symbol $\bp$ denotes the free additive convolution. Let $\omega$ be the standard semicircle law with mean zero and variance one:
\[
 \omega(dx)=\frac{1}{2\pi}\sqrt{4-x^2}\,\ind_{[-2,2]}(x)\,dx.
\]
For probability measures $\mu$ and $\nu$, their Kolmogorov distance is defined by
\[
 \Delta(\mu,\nu):=\sup_{x\in\R}
 \abs{\mu(( -\infty,x])-\nu(( -\infty,x])}.
\]
We also write $F_\mu(x):=\mu(( -\infty,x])$ for the distribution function of $\mu$.

The rate of convergence in the free central limit theorem has been studied in detail by Chistyakov--G\"otze \cite{CG} and in later works. 
In particular, consider freely independent self-adjoint random variables that are not necessarily identically distributed, and have mean zero, variances $\sigma_j^2>0$, and finite third absolute moments. For their normalized partial sum, Chistyakov--G\"otze \cite[Theorem~2.6]{CG} obtained a bound given by {\it the square root} of the third Lyapunov fraction,
\[
 L_{3,n}=B_n^{-3}\sum_{j=1}^n\int_{\R}|x|^3\,\mu_j(dx)
\]
In the identically distributed case, another argument in the same paper gives the optimal order $n^{-1/2}$.

Maejima--Sakuma \cite{MS} used truncation to obtain rates in the free central limit theorem under weaker moment conditions. On the other hand, Neufeld \cite[Theorem~1.7]{Neufeld2024} considered bounded, freely independent variables that are not necessarily identically distributed. If $T_j$ is a support radius of the $j$th variable, her result gives the linear bound
\[
 \Delta(\mu^{(n)},\omega)
 \le C B_n^{-3}\sum_{j=1}^n T_j^3.
\]
Neufeld \cite{Neufeld2025} also used the subordination method under a finite fourth-moment assumption. For every $\varepsilon\in(0,1/2)$, she obtained a bound using the fourth Lyapunov fraction raised to the power $1/2-\varepsilon$.

The purpose of this paper is to prove the following theorem.
This estimate has {\it exactly} the same form as the Berry--Esseen estimate in classical probability theory.

\begin{theorem}[Main theorem]\label{thm:main}
Let $\mu_1,\ldots,\mu_n\in\Pp(\R)$ satisfy that 
$\int_{\R}x\,\mu_j(dx)=0$,
and $\sigma_j^2:=\int_{\R}x^2\,\mu_j(dx)\in(0,\infty)$, 
and write $B_n^2:=\sum_{j=1}^n\sigma_j^2$.
Let $0<\delta\leq 1$ and
assume that $\beta_{2+\delta}(\mu_j)<\infty$.
Then, there exists an absolute constant $C>0$ such that
\begin{equation}\label{eq:two-plus-delta}
 \Delta(\mu^{(n)},\omega)
 \le  \frac{C}{B_n^{2+\delta}}\sum_{j=1}^n
 \int_{\R}|x|^{2+\delta}\,\mu_j(dx).
\end{equation}
\end{theorem}

To prove Theorem~\ref{thm:main} for any $\delta\in (0,1]$, we first show
the case $\delta=1$, namely, the following theorem.
\begin{theorem}\label{thm:main-1}
\begin{equation}\label{eq:3}
\Delta(\mu^{(n)},\omega)\le CL_{3,n}\left(=\frac{C}{B_n^3}\sum_{j=1}^n\int_{\R}|x|^3\,\mu_j(dx)\right).
\end{equation}
\end{theorem}

Sections 2--5 prepare the proof of Theorem~\ref{thm:main-1}, which is completed in Section 6, 
We then apply Theorem~\ref{thm:main-1} via a truncation argument in Section 7 to prove Theorem~\ref{thm:main}
for other general $\delta\in (0,1)$.

The proof of Theorem~\ref{thm:main-1} has four main parts: a truncation argument for the tails, a uniform estimate for the $R$-transforms based on a quantitative form of the inverse-function method of Benaych--Georges, a global analysis of a perturbed semicircle equation, and a Bai-type smoothing inequality. In particular, we do not use subordination functions. 

If $\mu_j$ is supported in $[-T_j,T_j]$, then $\beta_3(\mu_j)\le T_j^3$. Thus, Theorem~\ref{thm:main-1} also gives the estimate in \cite[Theorem~1.7]{Neufeld2024}. The main new point is that we use the actual third absolute moments instead of the support radii, and we do not assume boundedness.

We now explain the main steps of the proof. First, we truncate each $\mu_j$ at a fixed small scale $\tau B_n$. The tail mass, the means after truncation, and the change in the total variance are all of order $O(L_{3,n})$. After truncation, centering, and normalization, every distribution is supported in the same small interval around zero. 
Next, we give a quantitative form of an idea of Benaych-Georges
\cite{BG}. More precisely, for every centered probability measure
\(\nu\) supported in the fixed small interval, we prove, on a fixed disk,
\[
  \left|R_\nu(w)-m_2(\nu)w\right|
  \leq C\beta_3(\nu)|w|^2,
\]
where
\[
  m_2(\nu):=\int_{\mathbb R}x^2\,\nu(dx),
  \qquad
  \beta_3(\nu):=\int_{\mathbb R}|x|^3\,\nu(dx).
\]
Using the additivity of the $R$-transform, we then add the remainder terms. 
For a probability measure $\eta$, write
\[
 G_\eta(z):=\int_{\R}\frac{1}{z-x}\,\eta(dx),
 \qquad z\in\Cp,
\]
for its Cauchy transform. 
The Cauchy transform $G$ of the truncated sum satisfies
\[
 z=\frac1{G(z)}+G(z)+e(G(z)),
 \qquad |e(w)|\le C L_{3,n}|w|^2.
\]
A comparison with the semicircle equation gives that for $z \in \{z\in\C \mid \im z\gtrsim L_{3,n}\}$,
\[
 \abs{G(z)-G_\omega(z)}
 \le \frac{C L_{3,n}}{|\sqrt{z^2-4}|}.
\]
Finally, we integrate this pointwise bound in the vertical direction and apply a Bai-type smoothing inequality.

The numerical constants in this paper are chosen to keep the proof simple. We do not try to optimize them.

\section{Preliminaries}\label{sec:prelim}

\subsection{Cauchy transforms and \texorpdfstring{$R$}{R}-transforms}

For $k = 1,2,\ldots$, when the integrals are finite, write
$
 m_k(\mu):=\int_{\R}x^k\,\mu(dx).
 $
For $\mu\in\Pp(\R)$, its Cauchy transform is
\[
 G_\mu(z):=\int_{\R}\frac{1}{z-x}\,\mu(dx),
 \qquad z\in\Cp.
\]
We have $G_\mu(\Cp)\subset\Cm$ and $|G_\mu(z)|\le(\im z)^{-1}$. If $\mu$ is compactly supported, then $G_\mu$ is one-to-one in a neighborhood of infinity. Let $K_\mu$ be its inverse there. The function
\[
 R_\mu(w):=K_\mu(w)-\frac1w
\]
is analytic near $w=0$. The $R$-transform linearizes free additive convolution:
\begin{equation}\label{eq:R-add}
 R_{\mu\bp\nu}(w)=R_\mu(w)+R_\nu(w).
\end{equation}
See \cite{V86,BV}.

The Cauchy transform of the standard semicircle law is
\begin{equation}\label{eq:Gomega}
 G_\omega(z)=\frac{z-s(z)}{2},
 \qquad
 s(z):=(z-2)^{1/2}(z+2)^{1/2}.
\end{equation}
We choose the branch of $s$ that is analytic on $\C\setminus[-2,2]$ and satisfies $s(z)\sim z$ as $z\to\infty$. Then
\begin{equation}\label{eq:semi-eq}
 G_\omega(z)^2-zG_\omega(z)+1=0,
 \qquad
 2G_\omega(z)-z=-s(z).
\end{equation}

\subsection{A truncation comparison for free convolution}

We use the following inequality. By Bercovici--Voiculescu \cite[Proposition~4.13]{BV}, for any probability measures $\mu,\mu',\nu,\nu'$,
\begin{equation}\label{eq:contraction}
 \Delta(\mu\bp\nu,\mu'\bp\nu')
 \le \Delta(\mu,\mu')+\Delta(\nu,\nu').
\end{equation}

\begin{lemma}[Truncation comparison]\label{lem:trunc-comparison}
Let $T>0$. For each $j$, let $\nu_j$ be the probability measure obtained from $\mu_j$ by moving all mass outside $[-T,T]$ to the point zero. Let $Y_j$ have distribution $\nu_j$, and set
\[
 \alpha_j:=\int x\,\nu_j(dx),
 \qquad
 \beta_j^2:=\int x^2\,\nu_j(dx)-\alpha_j^2,
\]
\[
M_n:=\sum_{j=1}^n\alpha_j,
 \qquad
N_n^2:=\sum_{j=1}^n\beta_j^2>0,
 \qquad
\Gamma_n:=\sum_{j=1}^n\mu_j([-T,T]^c).
\]
Let $\widetilde\nu_{n,j}$ be the distribution of $(Y_j-\alpha_j)/N_n$, and let
\[
\widetilde\nu^{(n)}:=\widetilde\nu_{n,1}\bp\cdots\bp\widetilde\nu_{n,n}.
\]
Then $N_n\le B_n$, and
\begin{equation}\label{eq:trunc-comparison}
 \Delta(\mu^{(n)},\omega)
 \le \Delta(\widetilde\nu^{(n)},\omega)
 +\Gamma_n+\frac{|M_n|}{\pi N_n}
 +\frac{2}{\pi}\left(\frac{B_n}{N_n}-1\right).
\end{equation}
\end{lemma}

\begin{proof}
We first have
\[
 N_n^2=\sum_j\left\{\int_{[-T,T]}x^2\,\mu_j(dx)-\alpha_j^2\right\}
 \le B_n^2.
\]
Thus, $N_n\le B_n$. Also,
$\Delta(\mu_j,\nu_j)\le\mu_j([-T,T]^c)$. By repeated use of \eqref{eq:contraction}, and because a dilation by the same positive number does not change the Kolmogorov distance,
\[
 \Delta\!\left(D_{B_n^{-1}}(\mu_1\bp\cdots\bp\mu_n),
 D_{B_n^{-1}}(\nu_1\bp\cdots\bp\nu_n)\right)
 \le\Gamma_n.
\]
We may realize the $Y_j$ as freely independent random variables with laws $\nu_j$. Put
\[
 Y_n:=\sum_jY_j,
 \qquad
 Z_n:=\frac{Y_n-M_n}{N_n}.
\]
Then $Z_n$ has law $\widetilde\nu^{(n)}$, and
\[
 \frac{Y_n}{B_n}=\frac{N_n}{B_n}Z_n+\frac{M_n}{B_n}.
\]
Set $p:=B_n/N_n\ge1$ and $q:=-M_n/N_n$. 
It follows that
\[
 F_{D_{B_n^{-1}}(\nu_1\bp\cdots\bp\nu_n)}(x)
 =F_{\widetilde\nu^{(n)}}(px+q).
\]
The density of the semicircle law is at most $1/\pi$. Hence
\[
 \sup_x|F_\omega(x+q)-F_\omega(x)|\le\frac{|q|}{\pi}.
\]
For $p\ge1$, the difference is zero when $|x|>2$. When $|x|\le2$, the interval between $x$ and $px$ has length at most $2|p-1|$. Therefore,
\[
 \sup_x|F_\omega(px)-F_\omega(x)|
 \le\frac{2}{\pi}|p-1|.
\]
The triangle inequality now gives \eqref{eq:trunc-comparison}.
\end{proof}

\subsection{A Bai-type smoothing inequality}

We use the following form of the smoothing inequality. It is a version of Bai's method \cite{Bai} for the semicircle law. See also \cite[Corollary~2.3]{GT} and \cite[Proposition~2.5]{Neufeld2024}.

\begin{lemma}[Bai-type smoothing inequality]\label{lem:bai}
Let $\nu\in\Pp(\R)$ and assume that
\[
 \int_{\R}|F_\nu(x)-F_\omega(x)|\,dx<\infty.
\]
Let $h\in(0,1)$ and let $\varepsilon,\rho,\gamma>0$ satisfy
\[
 \gamma:=\frac1\pi\int_{|x|<\rho}\frac{dx}{1+x^2}>\frac12,
 \qquad
 \varepsilon>2h\rho.
\]
Set $I_\varepsilon=[-2+\varepsilon/2,2-\varepsilon/2]$. Then
\begin{align}
 \Delta(\nu,\omega)
 \le C_\gamma\Bigg(&\int_{\R}|G_\nu(u+i)-G_\omega(u+i)|\,du \notag\\
 &+\sup_{u\in I_\varepsilon}\int_h^1
 |G_\nu(u+iv)-G_\omega(u+iv)|\,dv
 +\frac{4\rho^2h}{\pi}+\gamma\varepsilon^{3/2}\Bigg),
 \label{eq:bai}
\end{align}
where $C_\gamma=(\pi(2\gamma-1))^{-1}$.
\end{lemma}

\section{A quantitative \texorpdfstring{$R$}{R}-transform estimate for measures with small support}\label{sec:R}

Benaych--Georges \cite{BG} proved a relation between finite moments and Taylor expansions of the $R$-transform. For our truncated triangular array, we need a remainder bound that is uniform on a fixed disk. The next theorem gives this bound.

\begin{theorem}[Quantitative $R$-transform estimate]\label{thm:quant-R}
Fix $R_0>2$, and set
\[
 a_0:=\frac{1}{100R_0^2}.
\]
Let $\nu$ be a probability measure with mean zero and
$\supp\nu\subset[-a_0,a_0]$. Then $R_\nu$ has an analytic continuation to the disk $\{|w|<R_0\}$, and
\begin{equation}\label{eq:quant-R}
 |R_\nu(w)-m_2(\nu)w|
 \le13\,\beta_3(\nu)|w|^2,
 \qquad |w|<R_0.
\end{equation}
\end{theorem}

\begin{proof}
For brevity, write $m_2=m_2(\nu)$ and $\beta_3=\beta_3(\nu)$. The support assumption gives
\begin{equation}\label{eq:moment-rel}
 m_2\le a_0^2,
 \qquad
 \beta_3\le a_0^3,
 \qquad
 m_2^2\le m_4(\nu)\le a_0\beta_3.
\end{equation}
If $\beta_3=0$, then $\nu=\delta_0$, and the result is clear.

For $|z|\le2R_0$, define
\[
 H(z):=z\int_{\R}\frac{1}{1-zx}\,\nu(dx).
\]
Since $|zx|\le2R_0a_0<1/100$ and $\nu$ has mean zero,
\begin{equation}\label{eq:H-expansion}
 H(z)=z+m_2z^3+z^4I(z),
 \qquad
 I(z):=\int_{\R}\frac{x^3}{1-zx}\,\nu(dx),
 \qquad
 |I(z)|\le\frac32\beta_3.
\end{equation}
Let $\varphi(z):=H(z)-z$. Then
\begin{equation}\label{eq:phi-bound}
 |\varphi(z)|\le\kappa|z|^3,
 \qquad
 \kappa:=m_2+3R_0\beta_3\le2a_0^2.
\end{equation}

Fix $|w|\le R_0$. On the circle $|z|=2R_0$,
\[
 |z-w|\ge R_0,
 \qquad
 |\varphi(z)|\le\kappa(2R_0)^3<R_0.
\]
Rouch\'e's theorem shows that $H(z)-w$ has exactly one zero $z=z(w)$ in $|z|<2R_0$, counted with multiplicity. Thus, this zero is simple, and the holomorphic implicit function theorem shows that $z(w)$ is analytic for $|w|<R_0$. Moreover, from $z=w-\varphi(z)$ and \eqref{eq:phi-bound},
\[
 |z|\le|w|+\kappa(2R_0)^2|z|
 \le|w|+\frac12|z|.
\]
Hence
\begin{equation}\label{eq:z-w-bound}
 |z(w)|\le2|w|.
\end{equation}

When $w$ is sufficiently close to zero, the identities
$H(z)=G_\nu(1/z)$ and $H(z(w))=w$ give
\[
 K_\nu(w)=\frac1{z(w)}.
\]
For $w\ne0$, we have $z(w)\ne0$. Moreover, by
\eqref{eq:z-w-bound} and \eqref{eq:phi-bound},
\[
 z(w)-w=-\varphi(z(w))=O(w^3)
 \qquad (w\to0).
\]
Hence
\[
 \widehat R(w):=\frac1{z(w)}-\frac1w
 =\frac{w-z(w)}{wz(w)}
\]
has a removable singularity at $w=0$. Near zero,
$\widehat R(w)=R_\nu(w)$, and therefore $\widehat R$ is the
desired analytic continuation of $R_\nu$ to $\{|w|<R_0\}$.
We continue to denote this analytic continuation by $R_\nu$.

Moreover, from \eqref{eq:H-expansion} and $H(z(w))=w$, we have
\begin{equation}\label{eq:R-representation}
 R_\nu(w)
 =\frac1{z(w)}-\frac1w
 =\frac{m_2z(w)^2}{w}
  +\frac{z(w)^3I(z(w))}{w},
 \qquad 0<|w|<R_0.
\end{equation}
For the remaining estimates, assume that $w\ne0$ and write
$z=z(w)$; the case $w=0$ is immediate. Then
\[
 R_\nu(w)-m_2w
 =\frac{m_2(z^2-w^2)}{w}+\frac{z^3I(z)}{w}.
\]
By \eqref{eq:z-w-bound} and $z-w=-\varphi(z)$,
\[
 |z-w|\le8\kappa|w|^3,
 \qquad
 |z+w|\le3|w|.
\]
Using \eqref{eq:moment-rel}, we get
\begin{align*}
 \left|\frac{m_2(z^2-w^2)}{w}\right|
 &\le24m_2\kappa|w|^3
 \le48a_0\beta_3|w|^3
 \le\frac14\beta_3|w|^2,\\
 \left|\frac{z^3I(z)}{w}\right|
 &\le12\beta_3|w|^2.
\end{align*}
Combining these estimates gives \eqref{eq:quant-R}.
\end{proof}

\begin{remark}
The Taylor expansion in \cite{BG} is an asymptotic expansion as $w\to0$ for each fixed measure. In Theorem~\ref{thm:quant-R}, the small support assumption gives a quantitative estimate on a fixed disk that is common to the whole triangular array. This uniformity lets us add the remainder terms directly.
\end{remark}

\section{Truncation and the total \texorpdfstring{$R$}{R}-transform}\label{sec:trunc-R}

Use Theorem~\ref{thm:quant-R} with $R_0=3$, and fix
\begin{equation}\label{eq:constants-a-tau}
 a_0:=\frac1{900},
 \qquad
 \tau:=\frac{a_0}{2\sqrt2}=\frac{1}{1800\sqrt2}.
\end{equation}
In this section, write $\lambda:=L_{3,n}$. If $\lambda>\tau^3$, then
$\Delta(\mu^{(n)},\omega)\le1\le\tau^{-3}\lambda$. Thus, from now on, assume that
\begin{equation}\label{eq:small-lambda}
 \lambda\le\tau^3.
\end{equation}

For each $j$, let $I_n:=[-\tau B_n,\tau B_n]$, and let $\nu_j$ be obtained from $\mu_j$ by moving all mass on $I_n^c$ to zero. Let $Y_j$ have distribution $\nu_j$, and set
\[
 \alpha_j:=\int x\,\nu_j(dx),
 \qquad
 \beta_j^2:=\int x^2\,\nu_j(dx)-\alpha_j^2,
\]
\[
 M_n:=\sum_{j=1}^n\alpha_j,
 \qquad
 N_n^2:=\sum_{j=1}^n\beta_j^2,
 \qquad
 \Gamma_n:=\sum_{j=1}^n\mu_j(I_n^c).
\]
As shown in Lemma~\ref{lem:trunc-est} below, the assumption
\eqref{eq:small-lambda} implies
\[
 N_n^2\ge\frac12B_n^2>0.
\]
Thus, the following normalized measures are well defined.
Let $\widetilde\nu_{n,j}$ be the distribution of $(Y_j-\alpha_j)/N_n$, and define
\[
 \widetilde\nu^{(n)}:=\widetilde\nu_{n,1}\bp\cdots\bp\widetilde\nu_{n,n}.
\]

\begin{lemma}[Basic estimates after truncation]\label{lem:trunc-est}
Under \eqref{eq:small-lambda}, the following estimates hold:
\begin{align}
 &\Gamma_n\le\tau^{-3}\lambda,
 \qquad
 \sum_{j=1}^n|\alpha_j|\le\tau^{-2}B_n\lambda,
 \qquad
 \sup_j|\alpha_j|\le\tau^{-2}B_n\lambda,
 \label{eq:alpha-est}\\
 &\frac12B_n^2\le N_n^2\le B_n^2,
 \qquad
 0\le1-\frac{N_n}{B_n}\le2\tau^{-1}\lambda,
 \label{eq:N-est}\\
 &\supp\widetilde\nu_{n,j}\subset[-a_0,a_0]
 \quad (1\le j\le n),
 \qquad
 \sum_{j=1}^n m_2(\widetilde\nu_{n,j})=1,
 \label{eq:support-var}\\
 &\sum_{j=1}^n\beta_3(\widetilde\nu_{n,j})\le24\lambda.
 \label{eq:third-trunc}
\end{align}
\end{lemma}

\begin{proof}
Markov's inequality gives
\[
 \Gamma_n\le(\tau B_n)^{-3}\sum_j\beta_3(\mu_j)=\tau^{-3}\lambda.
\]
Since $\mu_j$ has mean zero,
\[
 \alpha_j=-\int_{I_n^c}x\,\mu_j(dx),
 \qquad
 |\alpha_j|\le(\tau B_n)^{-2}\beta_3(\mu_j).
\]
Summing over $j$ gives \eqref{eq:alpha-est}. Also,
\begin{align*}
 N_n^2
 &=B_n^2-\sum_j\int_{I_n^c}x^2\,\mu_j(dx)-\sum_j\alpha_j^2\\
 &\ge B_n^2\left(1-\tau^{-1}\lambda-\tau^{-4}\lambda^2\right)
 \ge\frac12B_n^2.
\end{align*}
Clearly, $N_n^2\le B_n^2$. Therefore,
\[
 1-\frac{N_n}{B_n}
 \le1-\frac{N_n^2}{B_n^2}
 \le\tau^{-1}\lambda+\tau^{-4}\lambda^2
 \le2\tau^{-1}\lambda.
\]
This proves \eqref{eq:N-est}.

Next,
\[
 \frac{\tau B_n+|\alpha_j|}{N_n}
 \le\sqrt2\left(\tau+\tau^{-2}\lambda\right)
 \le2\sqrt2\tau=a_0.
\]
Thus, the support statement in \eqref{eq:support-var} holds. The sum of the variances is exactly one by the definition of $N_n$. Finally, using
$|x-y|^3\le4(|x|^3+|y|^3)$, we obtain
\begin{align*}
 \sum_j\beta_3(\widetilde\nu_{n,j})
 &\le\frac4{N_n^3}
 \left(\sum_j\beta_3(\mu_j)+\sum_j|\alpha_j|^3\right).
\end{align*}
Moreover,
\[
 \sum_j|\alpha_j|^3
 \le\left(\sup_j|\alpha_j|\right)^2\sum_j|\alpha_j|
 \le\tau^{-6}B_n^3\lambda^3
 \le B_n^3\lambda.
\]
Together with $N_n^{-3}\le2^{3/2}B_n^{-3}$, this gives \eqref{eq:third-trunc}.
\end{proof}

\begin{proposition}[Perturbation estimate for the total $R$-transform]\label{prop:sum-R}
Set $C_1:=312$. Each $R_{\widetilde\nu_{n,j}}$ has an analytic continuation to $|w|<3$, and
\begin{equation}\label{eq:sum-R}
 \left|\sum_{j=1}^nR_{\widetilde\nu_{n,j}}(w)-w\right|
 \le C_1\lambda|w|^2,
 \qquad |w|<3.
\end{equation}
\end{proposition}

\begin{proof}
By Lemma~\ref{lem:trunc-est} and Theorem~\ref{thm:quant-R},
\begin{align*}
 \left|\sum_jR_{\widetilde\nu_{n,j}}(w)-w\right|
 &=\left|\sum_j\{R_{\widetilde\nu_{n,j}}(w)
 -m_2(\widetilde\nu_{n,j})w\}\right|\\
 &\le13|w|^2\sum_j\beta_3(\widetilde\nu_{n,j})
 \le312\lambda|w|^2.
\end{align*}
\end{proof}

\section{Comparison of the Cauchy transforms}\label{sec:Cauchy}

Let $G:=G_{\widetilde\nu^{(n)}}$, and set
\begin{equation}\label{eq:def-e}
 e(w):=\sum_{j=1}^nR_{\widetilde\nu_{n,j}}(w)-w,
 \qquad
 |e(w)|\le C_1\lambda|w|^2
 \quad(|w|<3).
\end{equation}
Also set
\begin{equation}\label{eq:C2}
 C_2:=128C_1.
\end{equation}

\begin{lemma}[Global perturbation equation]\label{lem:global-eq}
If $z\in\Cp$ and $\im z>8C_1\lambda$, then
\begin{equation}\label{eq:G-less-2}
 |G(z)|<2,
\end{equation}
and
\begin{equation}\label{eq:global-eq}
 z=\frac1{G(z)}+G(z)+e(G(z)).
\end{equation}
\end{lemma}

\begin{proof}
Each $\widetilde\nu_{n,j}$ is compactly supported, and so is $\widetilde\nu^{(n)}$. Let $K$ be the inverse branch of $G$ at infinity. By \eqref{eq:R-add} and uniqueness of analytic continuation,
\[
 \widetilde K(w):=\frac1w+w+e(w),
 \qquad 0<|w|<3,
\]
is an analytic continuation of $K$. Thus, when $\im z$ is large enough,
$\widetilde K(G(z))=z$.

Let $\Omega:=\{z\in\Cp:|G(z)|<3\}$, and let $V$ be the
connected component of $\Omega$ that contains all points in the
upper half-plane with sufficiently large imaginary part. For
$z\in\Cp$, we have
\[
 \im G(z)
 =-(\im z)\int_{\R}\frac{1}{|z-x|^2}\,
 \widetilde\nu^{(n)}(dx)<0.
\]
In particular, $G(z)\ne0$. Therefore, the function
$\widetilde K(G(z))-z$ is analytic on $V$. By the identity theorem,
it is zero on all of $V$.

Suppose now that \eqref{eq:global-eq} holds and that $w:=G(z)$ satisfies $|w|=2$. Since $\im w<0$,
\[
 \im\left(\frac1w+w\right)
 =\im w\left(1-\frac1{|w|^2}\right)
 =\frac34\im w<0.
\]
Also, $|e(w)|\le4C_1\lambda$. Therefore,
\begin{equation}\label{eq:circle-exclusion}
 |G(z)|=2\ \text{and \eqref{eq:global-eq} holds}
 \quad\Longrightarrow\quad
 \im z\le4C_1\lambda.
\end{equation}

Fix $x\in\R$. We have $|G(x+i)|\le1$. Assume that there
exists $y_1\in(8C_1\lambda,1)$ such that
$|G(x+iy_1)|\ge2$. By continuity, the set
\[
 \{y\in[y_1,1]:|G(x+iy)|=2\}
\]
is nonempty. Let $y_0$ be its largest element. Then
\[
 |G(x+iy_0)|=2,
 \qquad
 |G(x+iy)|<2
 \quad (y_0<y\le1).
\]
For $y>1$, we have $|G(x+iy)|\le1/y<1$. Consequently,
the vertical ray
\[
 \{x+iy:y\ge y_0\}
\]
is contained in $\Omega$ and connects $x+iy_0$ to the upper
part of $V$. Hence $x+iy_0\in V$, and
\eqref{eq:global-eq} holds at this point. Since
$y_0>8C_1\lambda$, this contradicts
\eqref{eq:circle-exclusion}.

Therefore,
\[
 |G(x+iy)|<2
 \qquad (8C_1\lambda<y\le1).
\]
For $y>1$, we already have $|G(x+iy)|<1$. This proves
\eqref{eq:G-less-2} in the whole region
$\{z\in\Cp:\im z>8C_1\lambda\}$. Moreover, the vertical ray
from each point of this region to infinity is contained in
$\Omega$. Hence the whole region is contained in $V$, and
\eqref{eq:global-eq} holds throughout it.
\end{proof}

\begin{proposition}[Pointwise comparison with the semicircle Cauchy transform]\label{prop:pointwise}
If $z\in\Cp$ and $\im z\ge C_2\lambda$, then
\begin{equation}\label{eq:pointwise}
 |G(z)-G_\omega(z)|
 \le\frac{16C_1\lambda}{|s(z)|}.
\end{equation}
\end{proposition}

\begin{proof}
By Lemma~\ref{lem:global-eq}, we have $|G|<2$. Multiplying \eqref{eq:global-eq} by $G$ gives
\begin{equation}\label{eq:perturbed-quadratic}
 G(z)^2-zG(z)+1=q(z),
 \qquad
 q(z):=-G(z)e(G(z)),
 \qquad
 |q(z)|\le8C_1\lambda.
\end{equation}
Let $d:=G-G_\omega$. Subtracting \eqref{eq:semi-eq} from \eqref{eq:perturbed-quadratic}, we get
\begin{equation}\label{eq:d-equation}
 d(z)^2-s(z)d(z)-q(z)=0.
\end{equation}

For every $z=x+iy\in\Cp$,
\begin{equation}\label{eq:s-lower}
 |s(z)|^2=|z-2|\,|z+2|\ge2y.
\end{equation}
Indeed, both factors are at least $y$, and the triangle inequality shows that at least one of them is at least $2$. Thus, if $y\ge C_2\lambda$,
\[
 \left|\frac{4q(z)}{s(z)^2}\right|
 \le\frac{16C_1}{C_2}=\frac18.
\]
Consider the half-plane
\[
 D:=\{z\in\Cp:\im z>C_2\lambda\},
\]
and set
\[
 r(z):=\frac{4q(z)}{s(z)^2}.
\]
On $D$, we have $|r(z)|<1/8$. Hence $1+r(z)$ belongs to
the disk $\{\zeta\in\C:|\zeta-1|<1/8\}$, which is contained
in the right half-plane. We therefore choose the analytic branch
of $\sqrt{1+r(z)}$ characterized by
\[
 \Re\sqrt{1+r(z)}>0.
\]

The two analytic roots of \eqref{eq:d-equation} on $D$ are
\[
 d_\pm(z)=\frac{s(z)}2
 \left(1\pm\sqrt{1+r(z)}\right).
\]
By rationalizing the expression for the small root, we have
\[
 d_-(z)
 =-\frac{2q(z)}
 {s(z)\left(1+\sqrt{1+r(z)}\right)}.
\]
Since $\Re\sqrt{1+r(z)}>0$, we have
$|1+\sqrt{1+r(z)}|\ge1$. Therefore,
\begin{equation}\label{eq:small-root}
 |d_-(z)|\le\frac{2|q(z)|}{|s(z)|}.
\end{equation}
On the other hand,
\[
 |d_+(z)|
 =\frac{|s(z)|}{2}|1+\sqrt{1+r(z)}|
 \ge\frac{|s(z)|}{2}.
\]
Moreover,
\[
 |d_-(z)|
 \le\frac{|s(z)|}{2}|r(z)|
 <\frac{|s(z)|}{16}
 <\frac{|s(z)|}{4}.
\]
Thus, the two roots are separated on $D$.

As $Y\to\infty$,
$G(iY)\sim(iY)^{-1}$ and $G_\omega(iY)\sim(iY)^{-1}$,
so $d(iY)$ is the small root $d_-(iY)$. Since $D$ is connected
and the two roots are separated, continuity implies that
$d=d_-$ throughout $D$. The same conclusion holds on the boundary
$\im z=C_2\lambda$ by continuity. From
\eqref{eq:small-root} and \eqref{eq:perturbed-quadratic},
\[
 |d(z)|\le\frac{2|q(z)|}{|s(z)|}
 \le\frac{16C_1\lambda}{|s(z)|}.
\]
This proves \eqref{eq:pointwise}.
\end{proof}

\begin{lemma}[Integral estimates for smoothing]\label{lem:integrals}
If $C_2\lambda\le v\le1$, then, for every $x\in\R$,
\begin{equation}\label{eq:vertical}
 \int_v^1|G(x+iy)-G_\omega(x+iy)|\,dy
 \le16\sqrt2\,C_1\lambda.
\end{equation}
Also,
\begin{equation}\label{eq:horizontal}
 \int_{\R}|G(u+i)-G_\omega(u+i)|\,du
 \le100C_1\lambda.
\end{equation}
\end{lemma}

\begin{proof}
By \eqref{eq:s-lower} and Proposition~\ref{prop:pointwise},
\[
 \int_v^1|G(x+iy)-G_\omega(x+iy)|\,dy
 \le16C_1\lambda\int_0^1\frac{dy}{\sqrt{2y}}
 =16\sqrt2\,C_1\lambda.
\]
This proves \eqref{eq:vertical}.

Next, let $z=u+i$. Then
\[
 |s(u+i)|\ge2.
\]
Indeed,
$|s(u+i)|^4=|(u+i)^2-4|^2=(u^2-5)^2+4u^2\ge16$.
Thus, on $|u|\le5$, Proposition~\ref{prop:pointwise} shows that the contribution to the integral is at most $80C_1\lambda$.

Since $C_2\lambda\le1$ and $C_2=128C_1$, we have
$4C_1\lambda\le1/32<1$.
For $|u|>5$, \eqref{eq:global-eq} and $|G|<2$ give
\[
 \frac1{|G(u+i)|}
 \ge |u|-2-4C_1\lambda
 \ge |u|-3.
\]
Therefore,
\[
 |q(u+i)|\le\frac{C_1\lambda}{(|u|-3)^3}.
\]
Also, $|s(u+i)|^2\ge u^2-3\ge0.8u^2$. Hence, using \eqref{eq:small-root},
\begin{align*}
 \int_{|u|>5}|G(u+i)-G_\omega(u+i)|\,du
 &\le \frac{4C_1\lambda}{\sqrt{0.8}}
 \int_5^\infty\frac{du}{u(u-3)^3}\\
 &\le20C_1\lambda.
\end{align*}
Combining the two parts proves \eqref{eq:horizontal}.
\end{proof}

\begin{remark}
Near the endpoints $\pm2$, the factor $|s(z)|^{-1}$ becomes large. However, if we use \eqref{eq:s-lower} inside the integral, then
$\int_v^1y^{-1/2}\,dy$ is bounded uniformly in $v$. Therefore, the square-root singularity in the pointwise estimate near the endpoints does not cause any additional loss in $L_{3,n}$ in the vertical integral of the Bai-type inequality.
\end{remark}

\section{Proof of Theorem~\ref{thm:main-1}}

Set $\lambda=L_{3,n}$. If $\lambda>\tau^3$, the result is trivial, as explained above. Thus, assume \eqref{eq:small-lambda}.

Apply Lemma~\ref{lem:bai} to $\nu=\widetilde\nu^{(n)}$. Since $\widetilde\nu^{(n)}$ is compactly supported, it satisfies the integrability assumption in that lemma. Choose
\[
 h:=C_2\lambda,
 \qquad
 \rho:=2,
 \qquad
 \gamma:=\frac1\pi\int_{|x|<2}\frac{dx}{1+x^2}
 =\frac{2}{\pi}\arctan2>\frac12,
 \qquad
 \varepsilon:=5h.
\]
By our choice of constants and \eqref{eq:small-lambda}, we have $h\in(0,1)$ and $\varepsilon>2h\rho$. Lemma~\ref{lem:integrals} and \eqref{eq:bai} give
\begin{align*}
 \Delta(\widetilde\nu^{(n)},\omega)
 &\le C_\gamma\left(
 100C_1\lambda+16\sqrt2\,C_1\lambda
 +\frac{16h}{\pi}+\gamma(5h)^{3/2}
 \right)\\
 &\le C\lambda.
\end{align*}
Here we used $h=C_2\lambda<1$.

Finally, apply Lemma~\ref{lem:trunc-comparison} with $T=\tau B_n$. By Lemma~\ref{lem:trunc-est},
\[
 \Gamma_n\le\tau^{-3}\lambda,
 \qquad
 \frac{|M_n|}{N_n}\le\sqrt2\,\tau^{-2}\lambda,
\]
and
\[
 \frac{B_n}{N_n}-1
 =\frac{B_n}{N_n}\left(1-\frac{N_n}{B_n}\right)
 \le2\sqrt2\,\tau^{-1}\lambda.
\]
It follows that
\[
 \Delta(\mu^{(n)},\omega)
 \le C\lambda.
\]
Since $\tau$ is a fixed absolute constant, $C$ is also an absolute constant.

%\begin{proof}[Proof of Corollary~\ref{cor:iid}]
%We have $B_n^2=n\sigma^2$ and
%$\sum_{j=1}^n\beta_3(\mu_j)=n\beta_3(\mu)$. Hence
%\[
% L_{3,n}=\frac{\beta_3(\mu)}{\sigma^3\sqrt n}.
%\]
%The result follows from Theorem~\ref{thm:main}.
%\end{proof}

\begin{remark}[Comparison of the proof methods]
In this paper, we add the $R$-transforms of the truncated summands directly on a common disk. This gives one perturbed quadratic equation for the Cauchy transform of the truncated sum. Thus, we do not need separate estimates for subordination functions or an iterative improvement of exponents, which are important in \cite{Neufeld2024,Neufeld2025}. The third absolute moments control both the truncation error and the $R$-transform remainder through the same linear quantity $L_{3,n}$.
\end{remark}

\section{Proof of Theorem~\ref{thm:main}}

Let us recall the idea in our paper \cite{MS}. There, to obtain Berry--Esseen-type estimates under moment assumptions weaker than finiteness of the third moment, we used Theorem~2.6 in \cite{CG}, which states that, if $\beta_3(\mu_j)<\infty$ for all $j=1,2,\ldots,n$, then there exists an absolute constant $C>0$ such that
\[
 \Delta(\mu^{(n)},\omega)\le C L_{3,n}^{1/2}.
\]
However, we have now proved the better estimate in Theorem~\ref{thm:main-1}:
\begin{equation}\label{eq:section7-main}
 \Delta(\mu^{(n)},\omega)\le C L_{3,n}.
\end{equation}

In this section, unless a stronger moment condition is stated, we assume only that the measures $\mu_1,\ldots,\mu_n$ have mean zero and finite variances, with $B_n^2:=\sum_{j=1}^n m_2(\mu_j)>0$, and we define $\mu^{(n)}$ as in Section~\ref{sec:intro}.
As in Theorem~3.1 of \cite{MS}, let
\begin{equation}\label{eq:Lambda-ell}
 \Lambda_n:=\frac1{B_n^2}\sum_{j=1}^n
 \int_{|x|>\tau B_n}x^2\,\mu_j(dx),
 \qquad
 \Upsilon_n:=\frac1{B_n^3}\sum_{j=1}^n
 \int_{|x|\le\tau B_n}|x|^3\,\mu_j(dx),
\end{equation}
where $\tau$ is the fixed absolute constant defined in \eqref{eq:constants-a-tau}.

We first show the main part of the proof of the theorem as a lemma.

\begin{lemma}\label{lem:weaker-moment}
There exists an absolute constant $C>0$ such that
\begin{equation}\label{eq:Lambda-ell-bound}
 \Delta(\mu^{(n)},\omega)\le C(\Lambda_n+\Upsilon_n).
\end{equation}
\end{lemma}

\begin{proof}

Use the truncation interval $I_n=[-\tau B_n,\tau B_n]$ and define $\nu_j$, $\alpha_j$, $M_n$, $N_n$, and $\Gamma_n$ as in Section~\ref{sec:trunc-R}.
When $N_{n}>0$, define $\widetilde\nu_{n,j}$ and $\widetilde\nu^{(n)}$ as in that section. 
We separate two cases.

First suppose that $N_n^2\le B_n^2/4$. Since
\[
 B_n^2-N_n^2
 =\sum_{j=1}^n\int_{I_n^c}x^2\,\mu_j(dx)
  +\sum_{j=1}^n\alpha_j^2
 \le 2\sum_{j=1}^n\int_{I_n^c}x^2\,\mu_j(dx)
 =2B_n^2\Lambda_n,
\]
where $\alpha_j^2\le\int_{I_n^c}x^2\,\mu_j(dx)$, we obtain $\Lambda_n\ge3/8$. Hence
\[
 \Delta(\mu^{(n)},\omega)\le1\le\frac83\Lambda_n.
\]
This case also covers the possibility $N_n=0$.

Next suppose that $N_n^2>B_n^2/4$. Then $N_n>B_n/2$, and Lemma~\ref{lem:trunc-comparison} is applicable. Put
\[
 \Delta_n:=\Delta(\widetilde\nu^{(n)},\omega).
\]
After omitting any degenerate summands, Theorem~\ref{thm:main-1} applied to the centered and normalized truncated measures gives
\begin{align*}
 \Delta_n
 &\le \frac{C}{N_n^3}\sum_{j=1}^n
       \int_{\R}|x-\alpha_j|^3\,\nu_j(dx)\\
 &\le \frac{8C}{N_n^3}\sum_{j=1}^n
       \int_{I_n}|x|^3\,\mu_j(dx)
 \le 64C\Upsilon_n.
\end{align*}
Here we used
$\int|x-\alpha_j|^3\,\nu_j(dx)\le8\int|x|^3\,\nu_j(dx)$.
Moreover,
\begin{align}
 \Gamma_n&\le\tau^{-2}\Lambda_n,\notag\\
 \frac{|M_n|}{N_n}&\le 2\tau^{-1}\Lambda_n,\notag\\
 \frac{B_n}{N_n}-1
 &=\frac{B_n^2-N_n^2}{N_n(B_n+N_n)}
 \le\frac{2(B_n^2-N_n^2)}{B_n^2}
 \le4\Lambda_n.
 \label{eq:section7-trunc-errors}
\end{align}
The first two estimates follow from $|x|>\tau B_n$ on $I_n^c$, and the last one also uses
$B_n^2-N_n^2\le2B_n^2\Lambda_n$.
Substituting these estimates into \eqref{eq:trunc-comparison} proves \eqref{eq:Lambda-ell-bound}.
\end{proof}

We are now on the final stage of the proof of Theorem~\ref{thm:main}.

\begin{proof}
We have
\begin{align*}
 \Lambda_n
 &=\frac1{B_n^{2+\delta}}\sum_{j=1}^n
 \int_{|x|>\tau B_n}|x|^{2+\delta}
 \frac{B_n^\delta}{|x|^\delta}\,\mu_j(dx)\\
 &\le\frac{\tau^{-\delta}}{B_n^{2+\delta}}
 \sum_{j=1}^n\beta_{2+\delta}(\mu_j),
\end{align*}
and
\begin{align*}
 \Upsilon_n
 &=\frac1{B_n^{2+\delta}}\sum_{j=1}^n
 \int_{|x|\le\tau B_n}|x|^{2+\delta}
 \left(\frac{|x|}{B_n}\right)^{1-\delta}\,\mu_j(dx)\\
 &\le\frac{\tau^{1-\delta}}{B_n^{2+\delta}}
 \sum_{j=1}^n\beta_{2+\delta}(\mu_j).
\end{align*}
Combining these estimates with Lemma~\ref{lem:weaker-moment} proves \eqref{eq:two-plus-delta} with a positive constant $C=\tau^{-1}+1$, because $\tau^{-\delta}+\tau^{1-\delta}\le\tau^{-1}+1$.
\end{proof}

\begin{remark}\label{rem:iid-two-plus-delta}
If $0<\delta\leq 1$, $\mu_1=\cdots=\mu_n=\mu$, $m_2(\mu)=\sigma^2>0$, and $\beta_{2+\delta}(\mu)<\infty$, then Theorem~\ref{thm:main} gives
\[
 \Delta\!\left(D_{(\sigma\sqrt n)^{-1}}\mu^{\bp n},\omega\right)
 \le C\frac{\beta_{2+\delta}(\mu)}{\sigma^{2+\delta}n^{\delta/2}}.
\]
\end{remark}

%%%%%%%%%%%%%%%%%%%%%%%%%%%%%%%%%%%%%%%%%
%%%%%%%%%%%%%%%%%%%%%%%%%%%%%%%%%%%%%%%%%%%%


\begin{thebibliography}{99}

\bibitem{Bai}
Z.~D. Bai,
\newblock Convergence rate of expected spectral distributions of large random matrices. I. Wigner matrices,
\newblock \emph{Ann. Probab.} \textbf{21} (1993), 625--648.

\bibitem{BG}
F. Benaych-Georges,
\newblock Taylor expansions of $R$-transforms, application to supports and moments,
\newblock \emph{Indiana Univ. Math. J.} \textbf{55} (2006), 465--482.

\bibitem{BV}
H. Bercovici and D. Voiculescu,
\newblock Free convolution of measures with unbounded support,
\newblock \emph{Indiana Univ. Math. J.} \textbf{42} (1993), 733--773.

\bibitem{CG}
G.~P. Chistyakov and F. G\"otze,
\newblock Limit theorems in free probability theory. I,
\newblock \emph{Ann. Probab.} \textbf{36} (2008), 54--90.

\bibitem{GT}
F. G\"otze and A. Tikhomirov,
\newblock Rate of convergence to the semi-circular law,
\newblock \emph{Probab. Theory Related Fields} \textbf{127} (2003), 228--276.

\bibitem{MS}
M. Maejima and N. Sakuma,
\newblock Rates of convergence in the free central limit theorem,
\newblock \emph{Statist. Probab. Lett.} \textbf{197} (2023), 109802.

\bibitem{Neufeld2024}
L. Neufeld,
\newblock Weighted sums and Berry--Esseen type estimates in free probability theory,
\newblock \emph{Probab. Theory Related Fields} \textbf{190} (2024), 803--879.

\bibitem{Neufeld2025}
L. Neufeld,
\newblock A new Berry--Esseen-type estimate in the free central limit theorem,
\newblock arXiv:2503.23403, 2025.

\bibitem{V86}
D. Voiculescu,
\newblock Addition of certain non-commuting random variables,
\newblock \emph{J. Funct. Anal.} \textbf{66} (1986), 323--346.

\end{thebibliography}
\end{document}